\documentclass[12pt,a4paper,final]{amsart}
\usepackage{version}
\usepackage[notcite,notref]{showkeys}
\usepackage{amssymb}
\usepackage{graphicx}

\input epsf

\newtheorem{theorem}{Theorem}[section]
\newtheorem{lemma}[theorem]{Lemma}

\newtheorem{proposition}[theorem]{Proposition}

\theoremstyle{definition}

\theoremstyle{remark}

\numberwithin{equation}{section}

\DeclareMathOperator{\dimH}{dim_H}

\newcommand{\pr}{\mathbb{R}}
\newcommand{\pz}{\mathbb{Z}}

\newcommand{\de}{\delta}
\newcommand{\la}{\lambda}
\newcommand{\om}{\omega}
\newcommand{\si}{\sigma}
\newcommand{\Si}{\Sigma}
\newcommand{\wSi}{\widetilde{\Sigma}}
\newcommand{\wsi}{\widetilde{\sigma}}
\newcommand{\e}{\varepsilon}

\begin{document}

\title[2-dimensional measure with singular projection]
{Singularity of projections of 2-dimensional measures invariant under the 
  geodesic flow}

\author[R. Hovila]{Risto Hovila$^1$}
\address{Department of Mathematics and Statistics, P.O.Box 68,
         00014 University of Helsinki, Finland$^1$}
\email{risto.hovila@helsinki.fi}

\author[E. J\"arvenp\"a\"a]{Esa J\"arvenp\"a\"a$^2$}
\address{Department of Mathematical Sciences,  P.O. Box 3000,
         90014 University of Oulu, Finland$^{2,3}$}
\email{esa.jarvenpaa@oulu.fi$^2$}

\author[M. J\"arvenp\"a\"a]{Maarit J\"arvenp\"a\"a$^3$}
\email{maarit.jarvenpaa@oulu.fi$^3$}

\author[F. Ledrappier]{Fran\c cois Ledrappier$^4$}
\address{LPMA, UMR 7599, Universit\'e Paris 6, 4, Place Jussieu, 75252 Paris 
         cedex 05, France$^4$} 
\email{fledrapp@nd.edu$^4$}

\thanks{We acknowledge the Centre of Excellence in Analysis and Dynamics 
Research supported by the Academy of Finland. RH acknowledges the support from
Jenny and Antti Wihuri Foundation. FL acknowledges the partial
support from NSF Grant DMS-0801127.}


\begin{abstract} 
We show that on any compact Riemann surface with variable negative curvature 
there exists a measure which is invariant and ergodic under the geodesic flow 
and whose projection to the base manifold is 2-dimensional and singular 
with respect to the 2-dimensional Lebesgue measure.  
\end{abstract}

\maketitle

\section{Introduction}\label{intro}

Let $S$ be a compact surface with possibly variable negative curvature and let 
$\varphi=\varphi_t$, $t\in\pr$,
be the geodesic flow on the unit tangent bundle $T^1S$. In this note we are 
interested in the projections of $\varphi$-invariant measures to $S$ by the
canonical projection $\Pi:T^1S\to S$, where $\Pi(x,v)=x$. In 
particular, we prove:

\begin{theorem}\label{main} 
For any compact surface $S$ whose curvature is everywhere negative, there 
exists an ergodic $\varphi$-invariant measure $m$ on $T^1S$ such that $\Pi_*m$
has Hausdorff dimension equal to 2 and is singular with respect to 
the Lebesgue measure on $S$.
\end{theorem}

The image of a measure $\mu$ under a map $F$ is denoted by $F_*\mu$ and
the Hausdorff dimension of a measure $\mu$ is defined as follows:
\[
\dimH\mu=\inf\{\dimH A\mid\mu(A)>0, A\text{ is a Borel set}\}.
\]

It was shown by Ledrappier and Lindenstrauss \cite{LL} that the canonical 
projection of a $\varphi$-invariant measure of dimension greater than 
2 is absolutely continuous with respect to the Lebesgue measure. Our 
result shows that this property  does not hold at the threshold 2.  
Motivation for this study comes from Quantum Unique Ergodicity (QUE). Let 
$\psi_n$ be a sequence of orthonormal eigenfunctions of the Laplacian on $S$. 
The associated eigenvalues converge to infinity and the 
problem of  QUE is to describe the possible weak* limits of the probability 
measures with density $|\psi_n|^2$ as $n$ tends to infinity. This was solved by
Lindenstrauss for arithmetic hyperbolic surfaces in the case when $\psi_n$'s 
are also eigenfunctions of the Hecke operators -- the only limit is the 
normalized Lebesgue measure (see \cite{L}). In the case of a general 
hyperbolic surface, or of 
more general eigenfunctions, if any, on an arithmetic surface, Anantharaman and
Nonnenmacher \cite{AN} proved that any limit is the projection to $S$ of a 
$\varphi$-invariant measure with dimension at least 2. Rivi\`ere \cite{R} 
showed that this property is still true on surfaces with variable negative 
curvature. Our note shows that one still cannot conclude from their results 
that any weak* limit is nonsingular, which is a weak form of the QUE conjecture.

In order to prove Theorem \ref{main}, it suffices to construct an ergodic 
$\varphi$-invariant measure $m$ with dimension 2 and a measurable set 
$A\subset T^1S$ with the properties that $m(A^c)=0$ and
$\mathcal H^2(A)=0$, where $A^c=T^1S\setminus A$ and $\mathcal H^2$ is the 
2-dimensional Hausdorff measure. Since the Hausdorff measure cannot increase 
under the canonical projection, we have that $\Pi_\ast m$ is singular
with respect to the Lebesgue measure $\mathcal L^2$ on $S$. But, since the 
Hausdorff dimension of a $\varphi$-invariant measure is preserved under the
projection  by \cite{LL}, the Hausdorff dimension of $\Pi_\ast m$ is 2, as 
claimed.  For 3-dimensional Anosov flows, metric balls can be approximated by
dynamical ones on the manifold, so it is possible, by controlling entropy and 
exponent, to construct many invariant Gibbs measures with dimension exactly 2, 
see Section~\ref{Gibbs}. What we have to ensure is that the fluctuations of the 
measures of the dynamical balls are large enough so that for $m$-almost every 
point $(x,v)\in T^1S$, we have
\begin{equation}\label{bigmass}
\limsup_{\e\to 0}\frac{m(B((x,v),\e))}{\e^2}=+\infty .
\end{equation}
As shown in Section~\ref{iterlog} this property follows from a vector valued
almost sure invariance principle for hyperbolic dynamical systems proved in
\cite{MN2}.

\section{Invariant measures of dimension 2}\label{Gibbs}

We start by recalling some well-known facts of geodesic flows on negatively
curved surfaces (see for example \cite{KH}). Let $m$ be a $\varphi$-invariant 
ergodic measure on $T^1S$. The (largest Lyapunov) exponent $\la(m)$ is defined 
by
\[
\la(m)=\lim_{t\to\infty}\frac 1t\ln\|D\varphi_t(x,v)\|.
\]
The limit exists for $m$-almost all $(x,v)\in T^1S$ by the invariance and it is
$m$-almost surely independent of $(x,v)$ by the ergodicity. In the case of 
constant curvature $-1$ we have $\la=1$. The entropy 
$h_m(\varphi)$ is a number, $0\leq h_m(\varphi)\leq\la(m)$, which measures the 
randomness of typical trajectories (see e.g.~\cite{W} for the precise 
definition). The Hausdorff dimension of $m$ is given by
\[
\dimH m=1+2\frac{h_m(\varphi)}{\la(m)}
\]
(see \cite{PeS}). It follows that $m$ has dimension 2, if and only if we 
have $h_m(\varphi)=\la(m)/2$ ($=1/2$ in the case of constant curvature $-1$).
On any family for which the ratio $h_m(\varphi)/\la(m)$ varies continuously 
from $0$ to $1$, there will be measures with dimension $2$. In order to have 
specific examples, we shall consider a special family: Markov measures in a 
symbolic coding of the geodesic flow.

Let $A=A_{ij}$, $1\leq i,j\leq n$, be an $n\times n$-matrix with entries $0$ or
$1$ and define a subshift of finite type $\Si\subset\{1,\dots,n\}^\pz$ as 
the set of sequences $\underline\om=(\om_k)_{k\in\pz}$ such that 
$A_{\om_k\om_{k+1}}=1$ for all $k\in\pz$. The metric on $\Si$ is given by some 
number $\theta$ with $0<\theta<1$: $d(\underline\om,\underline\om)=0$ and for 
$\underline\om\ne\underline\om'$, 
$d(\underline\om,\underline\om')=\theta^{n(\underline\om,\underline\om')}$, where 
$n(\underline\om,\underline\om')$ is the largest number with $\om_j=\om_j'$ for
all $j$ with $|j|<n(\underline\om,\underline\om')$, using the interpretation
$n(\underline\om,\underline\om')=0$ if $\omega_0\ne\omega'_0$. The left shift 
on $\Si$ is denoted by $\sigma$, that is, 
$\sigma(\underline\omega)_i=\omega_{i+1}$ for all $i\in\mathbb Z$. If
there is a positive $p\in\mathbb N$ such that all the 
entries of the matrix $A^p$ are positive, the shift $(\Si,\si)$ is 
topologically mixing. For a positive continuous function $r$ on $\Sigma$, we 
define the special flow $(\wSi_r,\wsi_t)$, $t\in\pr$, by translation on the 
second coordinate, where
\[
\wSi_r:=\{(\underline\om,s)\mid\underline\omega\in\Sigma, 
   0\leq s\leq r(\underline\om)\}/(\underline\om,r(\underline\om))\sim 
   (\si(\underline\om),0),
\]
that is, for $t\ge 0$ we have 
$\wsi_t(\underline\om,s)=(\sigma^k(\underline\om),u)$ where 
$u=t+s-\sum_{j=0}^{k-1}r(\sigma^j(\underline\omega))$ and $k$ is the unique 
natural number satisfying $0\le u<r(\sigma^k(\underline\omega))$, and similarly
for $t<0$. 
The following result is due to Ratner \cite{Ra} (see \cite{S1} for a more 
geometric description of the space $\Si$).

\begin{proposition}\label{coding} 
There exist a mixing subshift of finite type $(\Si,\si)$, a H\"older 
continuous function $r$ on $\Si$ and a H\"older continuous mapping 
$\pi:\widetilde\Si_r\to T^1S$ such that 
$\pi\circ\widetilde\si_t=\varphi_t\circ\pi$. The mapping $\pi$ is finite-to-one
and one-to-one outside a closed invariant set of smaller topological entropy.
\end{proposition}

We call an $n\times n$-matrix $P=P_{ij}$, $1\leq i,j\leq n$, a Markov 
matrix on $\Si$ if it is Markov (meaning that $P_{ij}\geq 0$ for all 
$i,j=1,\dots,n$ and $\sum_{j=1}^n P_{ij}=1$ for all $i=1,\dots,n$) and 
$P_{ij}>0$, if and only if $A_{ij}=1$. To a Markov matrix $P$ on $\Si $ is 
associated a unique 
$\si$-invariant probability measure $\mu_P$ on $\Si$ given by the formula
\begin{equation}\label{muP}
\mu_P([\underline\om]_{n_2}^{n_1})=v_{\omega_{n_2}}\prod_{i=n_2}^{n_1-1}
  P_{\omega_i\omega_{i+1}}=e^{\ln v_{\omega_{n_2}}+\sum_{i=n_2}^{n_1-1}
  G(\sigma^i(\underline\omega))},
\end{equation}
where $[\underline\om]_{n_2}^{n_1}$ is the cylinder of order $(n_2,n_1)$ 
containing $\underline\om$, i.e. 
\[
[\underline\om]_{n_2}^{n_1}:=\{\underline\om'\in\Sigma\mid\om'_i=\om_i 
  {\textrm{ for }}n_2\leq i\leq n_1\},
\]
$v$ is the unique left eigenvector with $\sum_{i=1}^nv_i=1$ corresponding to 
the eigenvalue $1$ of $P$ 
and $G(\underline\omega)=\ln P_{\om_0\om_1}$. The system $(\Si,\si,\mu_P)$ is 
mixing for all such $P$. Define a $\widetilde\si$-invariant  
probability measure $\widetilde\mu_P$ on $\widetilde\Si_r$ by 
\[
\widetilde\mu_P=\frac{\int_\Sigma\mathcal L|_{[0,r)}\,d\mu_P}
  {\int_\Si r\, d\mu_P}
\]
and set $m_P=\pi_\ast\widetilde\mu_P$. The measure $m_P $ is ergodic for 
$\varphi $.

The entropy $h_{m_P}(\varphi)$ is given by Abramov formula (see \cite{Ab}): 
\[
h_{m_P}(\varphi)=h_{\widetilde\mu_P}(\widetilde\si)=\frac{h_{\mu_P}(\si)}
  {\int_\Si r\, d\mu_P}.
\]
The exponent $\la(m_P)$ is given by 
\[
\la(m_P)=\frac{\int_\Si F^u\,d\mu_P}{\int_\Si r\,d\mu_P},
\]
where 
$F^u(\underline\om)=\ln\|D\varphi_{r(\underline\omega)}(\pi(\underline\om,0))
  (v^u)\|$
and $v^u\in T_{\pi(\underline\om,0)}^1S$ is tangent to the unstable manifold at
$\pi(\underline\om,0)$. Note that $F^u(\underline\om)$  
is the expansion in the unstable direction from 
$\pi(\underline\om,0)$ to $\pi(\underline\om,r(\underline\om))$ 
($F^u=r$ in the constant curvature $-1$ case). The function $F^u$ is H\"older 
continuous and positive on $\Si$. In the same way one defines
$F^s$ for the inverse flow $\varphi_{-t}$. 
Our aim is to verify that there are many Markov matrices $P$ on $\Si$ such that 
$h_{\mu _P}(\si)/\int_\Si F^u\,d\mu_P=1/2$ giving $\dimH m_P=2$. The real 
analyticity of the mapping $P\mapsto h_{\mu_P}(\si)/\int_\Si F^u\,d\mu_P$ follows
from \cite[Corollary 7.10]{Ru},
but we do not know a priori that $1/2$ is a possible value for 
$h_{\mu_P}(\si)/\int_\Si F^u\,d\mu_P$. Nevertheless, we claim:

\begin{proposition}\label{dim} 
There exist a mixing subshift of finite type $(\Si,\si)$, a H\"older 
continuous function $r$ on $\Si$, a H\"older continuous mapping 
$\pi:\widetilde\Si_r\to T^1S$ such that 
$\pi\circ\widetilde\si_t=\varphi_t\circ\pi$ and the set of
Markov matrices $P$ on $\Si$ for which $\dimH m_P=2$
contains a smooth submanifold having codimension equal to
one. The mapping $\pi $ is finite-to-one and one-to-one 
outside a closed invariant set of smaller topological entropy. 
\end{proposition}

\begin{proof} To find infinitely many suitable Markov matrices,
it is sufficient to find a coding system
$\Si$ such that $1/2$ is an interior point of the image set of the map 
$P\mapsto h_{\mu_P}(\si)/\int_\Si F^u\, d\mu_P$. 
Since the set of Markov matrices is connected, the image set of 
$h_{\mu_P}(\si)/\int_\Si F^u\, d\mu_P$ is an interval. Clearly, 
$\int_\Si F^u\, d\mu_P\geq\inf F^u$, whereas $h_{\mu _P}(\si)$ can be 
arbitrarily small. Therefore, the interval is $(0,a_1]$, where 
$a_1:=\max_P h_{\mu_P}(\si)/\int_\Si F^u\, d\mu_P$. If $a_1\le 1/2$, we can 
relabel $\Si$, using as alphabet those words $i_1\cdots i_\ell$ of 
length $\ell$ for which $A_{i_ki_{k+1}}=1$ for all $k=1,\dots,\ell-1$ and as a
new matrix 
\[
A_{i_1\cdots i_\ell,j_1\cdots j_\ell}=1 \; \iff \; j_1=i_2,\dots,
   j_{\ell-1}=i_\ell.
\]
Markov matrices in this new presentation define a bigger family of invariant 
measures, the $\ell$-step Markov measures in the initial presentation. We set 
$a_\ell:=\max_{P_\ell}h_{\mu_{P_\ell}}(\si)/\int_\Si F^u\, d\mu_{P_\ell}$, where 
$P_\ell$ varies among all Markov measures in the alphabet of length $\ell$.
We have to show that for $\ell$ large enough we have $a_\ell>1/2$. 
Indeed, $a_\ell\to 1$ as $\ell\to\infty$. This is true because the Liouville 
measure $m_0$ on $T^1S$, for which $h_{m_0}(\varphi)=\la(m_0)$ and 
$\dimH m_0=3$, is a Gibbs measure whose entropy can be approximated by the 
entropies of $\ell$-step  Markov measures with $\ell$ going to infinity.

Finally, we prove that the set of Markov matrices giving dimension 2 contains
a smooth submanifold having codimension equal to 1. Consider a line $L=[v,a]$ 
starting from a vertex $v$ of the simplex of all Markov matrices and ending at 
a maximum point $a$ of the dimension function $d$. Recall that $d(a)>2$. 
Since $d(v)=1$ (entropy equals 0) there are points on $L$ where $d$ equals 2.
Denote this level set on $L$ by $L_2$. 
If $d'(x)\ne 0$ for some $x\in L_2$ we have found a regular point of the 
dimension function and we are done. Suppose $d'(x)=0$ for all $x\in L_2$.
All points $x\in L_2$ cannot be local maximum points since $d(a)>2$. Let 
$x_0$ be the
smallest point in $L_2$ which is not a local maximum point. Then $x_0$ 
is not a local minimum point since $d(y)\le d(x_0)$ for all $y\le x_0$.
Thus for small $h>0$ we have $d(x_0+h)-d(x_0)=h^n+\mathcal O(h^{n+1})$ for some
odd $n$. Therefore, after a 
local coordinate transformation $x_0$ becomes a regular point. This coordinate
transformation is smooth outside the hyperplane $V$ perpendicular to $L$ at
$x_0$ and is an identity on $V$. If the level set $d=2$ contains a point 
outside $V$, we are done. If the level set is contained in $V$, it is a piece
of a hyperplane.    
\end{proof}

Observe in particular that, since $m_P$ is ergodic and the support of $m_P$ is
the whole space $T^1S$, $\pi$ is invertible $m_P$-almost everywhere.

\section{Fluctuations of  measures of balls}\label{iterlog}

Suppose we have chosen $P$ by Proposition~\ref{dim} in such a way that 
$\dimH m_P=2$. Then 
\[
\lim _{\e\to 0}\frac{\ln m_P(B((x,v),\e))}{\ln\e}=2
\]
for $m_P$-almost all $(x,v)\in T^1S$. In this section we show that we can 
choose $P$ in such a way that equation (\ref{bigmass}) also holds for 
$m_P$-almost all $(x,v)\in T^1S$. Thus the upper derivative 
$\overline D(m_P,\mathcal H^2,(x,v))=\infty$ for $m_P$-almost all 
$(x,v)\in T^1S$ implying that $m_P$ and $\mathcal H^2$ are mutually singular
(see \cite[Theorem 2.12]{Ma}). By the discussion in the introduction, this 
completes the proof of Theorem~\ref{main}. We have to estimate 
$m_P(B((x,v),\e))$ from below. 

\begin{lemma}\label{balls}  
There exists a constant $C$ such that, for $\e$ small enough, if the functions
$n_1=n_1(\underline\om,\e)$ and $n_2=n_2(\underline\om,\e)$ satisfy 
\begin{equation}\label{smallradius}
\sum_{k=0}^{n_1}F^u(\si^k(\underline\om))\ge -\ln\e+C\text{ and }
\sum_{k=0}^{n_2}F^s(\si^{-k}(\underline\om))\ge -\ln\e +C,
\end{equation}
then
\[
B(\pi(\underline\om,t),\e)\supset\pi([\underline\om]_{-n_2}^{n_1}\times
  [t-\e/4,t+\e/4])
\]
for all $(\underline\om,t)\in\widetilde\Si_r$.
\end{lemma}

\begin{proof} Recall that $\varphi$ is an Anosov flow on $T^1S$ and the map 
$\pi$, given by Proposition~\ref{coding}, is constructed using a Markov 
partition for the flow (see for example \cite{B,Ra2}). In particular, there is 
$\de>0$ such that at each point $v\in T^1S$ (for notational simplicity we omit
the base point $x\in S$ from $(x,v)\in T^1S$) the sets 
\begin{align*} 
W^s_\de(v):=\{w\in T^1S\mid &d(\varphi_t(w),\varphi_t(v))<\de\;\;
   \forall t\geq 0\text{ and }\\
 &\lim_{t\to\infty}d(\varphi_t(w),\varphi_t(v))=0\}\text{ and}\\
W^u_\de(v):=\{w\in T^1S\mid &d(\varphi_t(w),\varphi_t(v))<\de\;\;
   \forall t\leq 0\text{ and }\\
 &\lim_{t\to-\infty}d(\varphi_t(w),\varphi_t(v))=0\}
\end{align*}
are open 1-dimensional submanifolds, called a local stable and a local unstable
manifold, respectively. Moreover, if 
$w_1,w_2\in\pi([\underline\omega]_{-n_2}^{n_1}\times\{0\})$ there exist a 
unique $t$ with $|t|<\delta$ and a unique 
$w_3:=[w_1,w_2]\in W_\de^s(\varphi_t(w_1))\cap W_\de^u(w_2)$. Furthermore, the
projection along the flow to $\pi([\underline\omega]_{-n_2}^{n_1}\times\{0\})$
is Lipschitz continuous when 
$-r(\sigma^{-1}(\underline\omega))\le t\le r(\underline\omega)$. 

Let $\underline\omega_1,\underline\omega_2\in [\underline\omega]_{-n_2}^{n_1}$
and define $w_1=\pi(\underline\omega_1,0)$ and $w_2=\pi(\underline\omega_2,0)$.
Choose $-\delta<t<\delta$ such that 
$w_3=[w_1,w_2]\in W_\delta^s(\varphi_t(w_1))$. Since $F^u$ and $F^s$ are
H\"older continuous, the distance between $w_2$ and $w_3$ along the unstable
leaf is, up to a bounded constant, 
$\exp(-\sum_{k=0}^{n_1}F^u(\si^k(\underline\om)))$, and the distance between 
$w_3$ and $\pi(\underline\omega_1,t)$ along the stable leaf is, up to a bounded 
constant, $\exp(-\sum_{k=0}^{n_2}F^s(\si^{-k}(\underline\om)))$. Using the 
Lipschitz continuity of the projection along the flow, one can choose a 
constant $C$ in the assumptions such that $d(w_1,w_2)<\varepsilon$. 
Since the flow $\varphi_t$ is smooth there exists $c'>0$ such that 
$d(\varphi_u(w_1),\varphi_u(w_2))\le c'd(w_1,w_2)$ for all 
$0\le u\le r(\underline\omega)$. Finally, as 
$d(\pi(\underline\omega_2,t),\pi(\underline\omega_2,u))\le |t-u|$, the claim 
follows by changing the constant $C$ obtained above. 
\end{proof}

Let $(x,v)=\pi(\underline\omega,t)$. Using Lemma~\ref{balls} and \eqref{muP},
we find $c>0$ such that
\begin{equation}\label{lowerbound}
m_P(B((x,v),\e))\geq\frac\e2\mu_P([\underline\om]_{-n_2}^{n_1})\geq c\e 
  e^{\sum_{k=-n_2}^{n_1-1}G(\si^k(\underline\om))},
\end{equation}
where $n_1$ and $n_2$ are as in Lemma~\ref{balls}. Define
\begin{align*}
&X_n^u(\underline\omega)=-\sum_{i=0}^{n-1}G(\sigma^i(\underline\omega)),\quad
  X_n^s(\underline\omega)=-\sum_{i=1}^nG(\sigma^{-i}(\underline\omega)),\\
&Y_n^u(\underline\omega)=\sum_{i=0}^{n-1}F^u(\sigma^i(\underline\omega))\,
  \text{ and }\,
  Y_n^s(\underline\omega)=\sum_{i=0}^{n-1}F^s(\sigma^{-i}(\underline\omega)).
\end{align*}
By the Shannon-McMillan-Breiman theorem (see for example 
\cite[Remark p.~93]{W}) and by the choice of $P$, we have for $\mu_P$-almost 
all $\underline\omega\in\Sigma$ that
\begin{equation}\label{dimis1}
\lim_{n\to\infty}\frac{X_n^u}{Y_n^u}=\frac12=\lim_{n\to\infty}\frac{X_n^s}{Y_n^s}.
\end{equation}
Let $a=-\int_\Sigma G\,d\mu_P$ and 
$b=\int_\Sigma F^u\,d\mu_P=\int_\Sigma F^s\,d\mu_P$. Then for $v\in\{u,s\}$
\begin{equation}\label{ergodic}
\lim_{n\to\infty}\frac{X_n^v(\underline\omega)}n=a\text{ and }
 \lim_{n\to\infty}\frac{Y_n^v(\underline\omega)}n=b
\end{equation}
for $\mu_P$-almost all $\underline\omega\in\Sigma$. Equation \eqref{dimis1} 
gives $b=2a$.

Observe that $(X_n^u-na,Y_n^u-nb)$ and $(X_n^s-na,Y_n^s-nb)$ are not 
independent. However, using the fact that $F^u$ and $G$ are H\"older 
continuous, we can find $K>0$ such that for all $n\in\mathbb N$ we have
\begin{equation}\label{Holder}
|X_n^u(\underline\omega)-X_n^u(\underline\omega')|<K\text{ and }
|Y_n^u(\underline\omega)-Y_n^u(\underline\omega')|<K
\end{equation}
for all 
$\underline\omega,\underline\omega'\in\Sigma$ such that $\omega_i=\omega_i'$
for all $i=0,1,\dots$ Moreover, the same holds for $X_n^s$ and $Y_n^s$ with the
condition $\omega_i=\omega_i'$ for all $i=-1,-2,\dots$ Let $\Sigma^<$ and 
$\Sigma^\ge$ be the one-sided subshifts corresponding to the negative and 
nonnegative indices, respectively. Fix $\underline\xi_i^-\in\Sigma^<$ and
$\underline\xi_i^+\in\Sigma^\ge$ for $i=1,\dots,n$ such that 
$(\underline\xi_i^-)_{-1}=i$ and $(\underline\xi_i^+)_0=i$. Define 
$\eta,\psi:\{1,\dots,n\}\to\{1,\dots,n\}$ such that $A_{\eta(i)\,i}=1$ and
$A_{i\,\psi(i)}=1$. Setting
$X_n^u(\underline\omega^\ge):=X_n^u(\underline\xi_{\eta(\omega_0^\ge)}^-\vee
  \underline\omega^\ge)$,
$Y_n^u(\underline\omega^\ge):=Y_n^u(\underline\xi_{\eta(\omega_0^\ge)}^-\vee
  \underline\omega^\ge)$,
$X_n^s(\underline\omega^<):=X_n^s(\underline\omega^<\vee
  \underline\xi_{\psi(\omega_{-1}^<)}^+)$ and
$Y_n^s(\underline\omega^<):=Y_n^s(\underline\omega^<\vee
  \underline\xi_{\psi(\omega_{-1}^<)}^+)$,
one can consider $X_n^u$ and $Y_n^u$ as functions on $\Sigma^\ge$ and 
$X_n^s$ and $Y_n^s$ as functions on $\Sigma^<$. By \cite[Lemma 5.9]{Ru} the 
measure $\mu_P^<\times\mu_P^\ge$ restricted to the set
$\{(\underline\omega^<,\underline\omega^\ge)\in\Sigma^<\times\Sigma^\ge\mid
   \underline\omega^<\vee\underline\omega^\ge\in\Sigma\}$
is equivalent with $\mu_P$ and the 
Radon-Nikodym derivative is bounded from above and from below, where 
$\mu_P^<$ and $\mu_P^\ge$ are the Markov measures given by \eqref{muP} on
$\Sigma^<$ and $\Sigma^\ge$, respectively. In particular,
there exists $L>0$ such that
\begin{equation}\label{equivalent}
\mu_P([\underline\omega]_{-n_2}^{n_1})\ge L^{-1}(\mu_P^<\times\mu_P^\ge)
   ([\underline\omega]_{-n_2}^{-1}\times [\underline\omega]_0^{n_1})
\end{equation}
for all $n_1,n_2\in\mathbb N$. Further, inequalities \eqref{Holder} imply
that \eqref{smallradius} is valid for any $\underline\omega\in\Sigma$ for which
$Y_{n_1}^u(\underline\omega^\ge)\ge -\ln\varepsilon+C+K$ and
$Y_{n_2}^s(\underline\omega^<)\ge -\ln\varepsilon+C+K$. Thus it is enough to
consider $(X_n^u-na,Y_n^u-nb)$ and $(X_n^s-na,Y_n^s-nb)$ as independent 
observables.

The almost sure invariance principle \cite[Theorem 3.6]{MN2} (see also
\cite[Theorem 1.3]{MN2}) implies that the observables $(X_n^u-na,Y_n^u-nb)$ and 
$(X_n^s-na,Y_n^s-nb)$ can be approximated by 2-dimensional Brownian motions.
This means that for $v\in\{u,s\}$ there exist $\lambda>0$ and a probability 
space $(X,\mathbb P)$ supporting a sequence of random variables 
$(\widetilde X_n^v,\widetilde Y_n^v)$
having the same distribution as $(X_n^v-na,Y_n^v-nb)$ and a 2-dimensional 
Brownian motion $W^v$ with a covariance matrix $Q^v$ such that  
\[
|(\widetilde X_n^v,\widetilde Y_n^v)-W^v(n)|\ll n^{\frac 12-\lambda}
\]
$\mathbb P$-almost surely for large $n$. In fact, $Q^v$ does not depend on 
$v$ but this information is not necessary.

To prove that the covariance matrix $Q$ is nonsingular, we need a couple of 
lemmas. For a smooth flow on a compact manifold $M$ generated by a vector 
field $X$ and preserving a probability measure $\mu$, one can define an element
$H$ of the first homology space $H_1(M,\mathbb R)$ in the following way:
The mapping which associates to a closed $1$-form $\psi$ on $M$  the integral 
$\int_M\psi(X)\,d\mu$ is linear and, by invariance of $\mu$, vanishes on exact 
$1$-forms. This defines a linear real valued mapping on the first de Rham 
cohomology space $H_{dR}^1(M,\mathbb R)$. We denote it and the corresponding  
element of $H_1(M,\mathbb R)$ (given by de Rham's theorem) by $H$.
We need the following result of Arnol'd \cite[Lemma 23.2]{An}:

\begin{lemma}\label{Arnold}
With the above notation, for the geodesic flow on a compact manifold not 
homeomorphic to the two-torus and for the Liouville measure $\mu$, we have
$H=0$.
\end{lemma}

We will also use a variant of Liv\v{s}ic's Theorem \cite{Liv}:

\begin{lemma}\label{Livcic}
Let $F$ be a H\"older continuous positive function on $T^1S$ and 
$b$ a positive number such that $\int_\gamma F\in b\mathbb N$ for all closed 
geodesics $\gamma$. Then there is a H\"older continuous function $f$ on $T^1S$ 
with values in the complex unit circle such that for all $v\in T^1S$ and for
all $t\ge 0$
\begin{equation}\label{cocycle} 
f(\varphi_t(v))=e^{\frac{2i\pi }{b}\int_0^t F(\varphi_s(v))\,ds}f(v).
\end{equation}
\end{lemma}

\begin{proof} The proof is analogous to that given in \cite{Liv}. Let 
$v_0\in T^1S$ be a vector such that 
$\varphi_t(v_0)$, $t\geq 0$, is dense in $T^1S$. We first define $f$ on the 
orbit $\varphi_t(v_0)$ by 
\[
f(\varphi_t(v_0))= e^{\frac{2i\pi}{b}\int_0^t F(\varphi_s(v_0))\,ds}\text{ for }
t\ge 0.
\]
Clearly $f$ takes values in the unit circle and satisfies equation 
\eqref{cocycle} for all $v$ in the orbit of $v_0$ and for all $t\geq 0$. We 
claim that $f$ is uniformly H\"older continuous on the orbit of $v_0$. Since 
the orbit is dense, this implies that $f$ extends to a H\"older continuous 
function on $T^1S$. Obviously this extension has the desired properties. Now 
we prove the 
uniform H\"older continuity of $f$. If $\varphi_t(v_0)$ and $\varphi_{t'}(v_0)$
are close enough (assume that $t<t'$), there exist a uniform constant $c$ and 
a closed geodesic $\gamma$ of length $\ell$ with 
$|\ell-(t'-t)|< cd(\varphi_t(v_0), \varphi_{t'}(v_0))$ such that  
$d(\gamma(s),\varphi_{t+s}(v_0))\le cd(\varphi_t(v_0),\varphi_{t'}(v_0))$ for 
$0\le s\le\ell$ by Anosov Closing Lemma (see 
\cite[Theorem 6.4.15 and p. 548]{KH}). 
Moreover, since the two geodesics $\gamma(s)$, $0\le s\le\ell$, and 
$\varphi_s(v_0)$, $t\le s\le t'$, remain close, they have to get closer 
exponentially, i.e. there exist $\tau$ and $\tau'$ with 
$|\tau |,|\tau'|<c d(\varphi_t(v_0),\varphi_{t'}(v_0))$ and $\tilde c,a>0$ such 
that
\[
d(\varphi_{t+s}(v_0),\gamma(\tau+s))\le\tilde ce^{-as}
   d(\varphi_t(v_0),\varphi_{t'}(v_0))
\]
for $0\le s\le (t'-t)/2$ and
\[
d(\varphi_{t+s}(v_0),\gamma(\tau'+s))\leq\tilde ce^{-a(t'-t -s)}
  d(\varphi_t(v_0),\varphi_{t'}(v_0))
\]
for $(t'-t)/2\le s\le t'-t$ (see \cite[Corollary 6.4.17]{KH}). This implies 
that
\[
\left|\int_t^{t'} F(\varphi_s(v_0))\,ds-\int_0^\ell F(\gamma(s))\,ds\right|
 \le\bar c d(\varphi_t(v_0),\varphi_{t'}(v_0))^\alpha,
\]
where $\alpha $ is the H\"older exponent of $F$ and $\bar c>0$. Since 
$\int_0^\ell F(\gamma(s))\,ds$
is a multiple of $b$, we have for some $\hat c>0$ that
\[
|f(\varphi_t(v_0))-f(\varphi_{t'}(v_0))|=|e^{\frac{2i\pi}b\int_t^{t'} 
  F(\varphi_s(v_0))\,ds}-1|\le\hat c d(\varphi_t(v_0),\varphi_{t'}(v_0))^\alpha.
\]
\end{proof}

\begin{lemma}\label{nonsingular}
There exists $P$ such that $\dimH m_P=2$ and the covariance matrix $Q$ is 
nonsingular.
\end{lemma}

\begin{proof}
According to \cite[Remark 1.2]{MN2} (see also \cite[Section 4.3]{HM} and 
\cite[Corollary 2.3]{MN1}), the covariance matrix $Q$ is nonsingular, if
\[
-G-a+\alpha(F^v-b)\ne f\circ\sigma-f\text{ and } F^v-b\ne f\circ\sigma-f
\]
for all $\alpha\in\mathbb R$ and for all Lipschitz functions $f$. 
To check the first claim, it is enough to find a periodic orbit 
$\gamma=(\underline\omega_1,\dots,\underline\omega_n)$ such that
\begin{equation}\label{zerosum}
\sum_{i=1}^n(-G-a+\alpha(F^v-b))(\underline\omega_i)\ne 0.
\end{equation}

Choose periodic orbits
$\gamma_j=(\underline\omega_1^j,\dots,\underline\omega_{n_j}^j)$ for $j=1,2,3$
such that each $\gamma_j$ contains a step $k_jl_j$ (that is, 
$(\underline\omega_1^j)_{s,s+1}=k_jl_j$ for some $s$) which is not included in 
the other orbits. Assume that one can vary $P_{k_jl_j}$ without changing
$P_{mn}$ for any other step $mn$ included in $\gamma_j$ for $j=1,2,3$ such that
\eqref{dimis1} is valid. This is possible 
since by Proposition~\ref{dim} the level set contains a submanifold having 
codimension equal to 1 and we see in the proof of Proposition~\ref{dim} that we
may choose a coding system $\Sigma$ such that the 
dimension of the space of Markov matrices is as large as we wish. 
Write $\beta=a+\alpha b$ and consider the system of equations
\begin{equation}\label{firsttwo}
\left\{
\begin{aligned}
\sum_{i=1}^{n_1}(-G+\alpha F^v)(\underline\omega_i^1)&=n_1\beta\\
\sum_{i=1}^{n_2}(-G+\alpha F^v)(\underline\omega_i^2)&=n_2\beta.
\end{aligned}\right.
\end{equation}
View for a moment $\alpha$ and $\beta$ as independent variables. Assume that 
this linear system has a unique solution $(\alpha,\beta)$. If $(\alpha,\beta)$
is not a solution of 
\[
\sum_{i=1}^{n_3}(-G+\alpha F^v)(\underline\omega_i^3)=n_3\beta,
\]
we are done. If $(\alpha,\beta)$ is a solution for this third equation, we
may change $G$ in the third equation by varying $P_{k_3l_3}$ without changing 
$G$ in \eqref{firsttwo}.
If this change does not change $a$ and $b$, that is $\beta$, then the
solution $\alpha$ for the third equation will change 
and we are done. If $\beta$ changes, then $(\alpha,\beta)$ is not a solution
for \eqref{firsttwo} any more. The remaining case is that \eqref{firsttwo} has
infinitely many solutions. Then the equations describe the same line, and 
therefore changing $G$, that is, the affine part in the second equation, 
implies that we have two different parallel lines and there is no solution for 
\eqref{firsttwo}.

Finally, we have to show that there exists a periodic orbit
$\gamma=(\underline\omega_1,\dots,\underline\omega_n)$ such that   
\[
\sum_{i=1}^n(F^v-b)(\underline\omega_i)\ne 0.
\]
This follows from Lemma~\ref{Arnold}. Indeed, if this is not the case, let $f$
be the function given by Lemma~\ref{Livcic}. Then $\ln f$ is a purely 
imaginary multivalued function which is smooth in the direction of the flow and
$X(\ln f)=\frac{2i\pi }b F$. Given $\e$, one can find a  smooth function $g$ 
with values in the complex unit circle  such that $|X(\ln g) - X(\ln f)|<\e$ 
(see \cite[Proof of Lemma 23.1]{An}). Then $\frac 1i d\ln g$ is a real  
$1$-form on $T^1S$ such that
\[
H(\tfrac 1i d\ln g)=\int_M\tfrac 1i X(\ln g)\,d\mu\ne 0
\]
as soon as $\e<\frac{2\pi}b\int_M F\,d\mu$.
\end{proof}

Fix $D>1$ and let $\widetilde C=C+K$, where $C$ is as in Lemma~\ref{balls} and
$K$ is as in \eqref{Holder}. Define events $E_n^v$ for
$v\in\{u,s\}$ by ``$X_n^v\le na-D\sqrt n$ and $Y_n^v\ge nb+\widetilde C$'' and 
let $E$ be the event ``$E_n^u$ and $E_n^s$ happens for infinitely many $n$''. 
Then $E$ is a tail event, and therefore, it has probability 0 or 1. By 
independence,
\[
(\mu_P^<\times\mu_P^\ge)(E)\ge\limsup_{n\to\infty}\bigl(\mu_P^\ge(E_n^u)
  \bigr)^2.
\]
The scale invariance \cite[Lemma 1.7]{MP} and Lemma~\ref{nonsingular} imply 
that
\begin{align*}
\mathbb P&\bigl(W(n)\in\{(x,y)\mid x<-2D\sqrt n\text{ and }
  y>n^{\frac 12-\lambda}+\widetilde C\}\bigr)\\
  &\ge\mathbb P\bigl(W(1)\in\{(x,y)\mid x<-2D\text{ and }
  y>1\}\bigr)=\rho>0
\end{align*}
for $n>2\widetilde C^2$. Choose $n$ large enough such that 
\[
\mathbb P\bigl(|(\widetilde X_{n'}^u,\widetilde Y_{n'}^u)-W^u(n')|
  <(n')^{\frac 12-\lambda}\text{ for all }n'\ge n\bigr)>1-\frac\rho 2.
\]
The almost sure invariance principle implies that 
$\mu_P^\ge(E_n^u)>\rho/2$ and thus the event $E$ has probability $1$. 

Write $\varepsilon(n)=e^{-nb}$. Recalling \eqref{lowerbound}, one can 
find for $\mu_P$-almost all $\underline\omega\in\Sigma$ a sequence $n$
tending to infinity such that 
\[
m_P(B((x,v),\e(n)))\ge\tilde c\e(n)e^{-2an+2D\sqrt n}
   =\tilde c\e(n)^2e^{c'\sqrt{-\ln\varepsilon(n)}}
\]
for some constants $\tilde c$ and $c'$.
This implies \eqref{bigmass}.

\vskip 0.2truecm

\textbf{Acknowledgement} We thank Mark Pollicott for pointing us out the 
argument needed in the proof of the final part of Lemma~\ref{nonsingular}.  
We also want to mention the article \cite{PUZ} from where we got our original
inspiration for the whole proof.

\end{document}